\documentclass[aps,pra,floats,tightenlines,a4paper,showpacs,onecolumn,10pt,longbibliography,superscriptaddress,notitlepage]{revtex4-1}
\usepackage{bbm, amsmath, amssymb, amsthm, bm,textcomp, nicefrac,geometry,ragged2e}
\usepackage{graphicx,epstopdf,color,verbatim,enumitem}
\usepackage[dvipsnames]{xcolor}
\usepackage[bbgreekl]{mathbbol}
\usepackage{graphicx,epstopdf,color,verbatim,enumitem,ulem}
\definecolor{myurlcolor}{rgb}{0,0,0.7}
\definecolor{myrefcolor}{rgb}{0.8,0,0}
\usepackage[unicode=true,pdfusetitle, bookmarks=false,bookmarksnumbered=false,
bookmarksopen=false, breaklinks=false,pdfborder={0 0 0},backref=false,
colorlinks=true, linkcolor=myrefcolor,citecolor=myurlcolor,urlcolor=myurlcolor]{hyperref}
\usepackage{pbox,array}
\usepackage{tikz}
\usetikzlibrary{shapes.geometric}
\usepackage{diagbox,xcolor}

\newcommand{\be}{\begin{equation}}
\newcommand{\ee}{\end{equation}}
\newcommand{\bea}{\begin{eqnarray}}
\newcommand{\eea}{\end{eqnarray}}

\def\Tb{{\bar T}}
\def\qb{{\bar q}}
\def\Qm{{\normalfont Q}}
\def\Tm{{\normalfont S}}
\def\Qf{{\normalfont {\text Q}}}
\def\Tf{{\normalfont {\text S}}}

\makeatletter
\newtheorem*{rep@theorem}{\rep@title}
\newcommand{\newreptheorem}[2]{%
\newenvironment{rep#1}[1]{%
 \def\rep@title{#2 \ref{##1}}%
 \begin{rep@theorem}}%
 {\end{rep@theorem}}}
\makeatother

\newreptheorem{theorem}{Theorem}
\newtheorem{lemma}{Lemma}
\newtheorem{proposition}{Proposition}
\newtheorem{theorem}{Theorem}
\newtheorem{corollary}{Corollary}

\newtheorem*{result*}{Result}

\begin{document}

\title{Simple Buehler-optimal confidence intervals on the average success probability of independent Bernoulli trials}

\author{Jean-Daniel Bancal}
\affiliation{Universit\'e Paris-Saclay, CEA, CNRS, Institut de physique th\'eorique, 91191, Gif-sur-Yvette, France}
\author{Pavel Sekatski}
\affiliation{Département de Physique Appliquée, Université de Genève, 1211 Genève, Suisse}

\date{\today}

\begin{abstract}
One-sided confidence intervals are presented for the average of non-identical Bernoulli parameters. These confidence intervals are expressed as analytical functions of the total number of Bernoulli games won, the number of rounds and the confidence level. Tightness of these bounds in the sense of Buehler, i.e.~as the strictest possible monotonic intervals, is demonstrated for all confidence levels. A simple interval valid for all confidence levels is also provided with a tightness guarantee. Finally, an application of the proposed confidence intervals to sequential sampling is discussed.
\end{abstract}

\maketitle



\section{Introduction}

Let $T_1, T_2, \ldots T_n=0,1$ be independent Bernoulli random variables with possibly distinct first moments $P(T_i=1)=q_i\in[0,1]$. The average success probability of these $n$ independent binary games, 
\begin{equation}
\bar q=\frac{1}{n} \sum_{i=1}^n q_i,
\end{equation}
is an important quantity. Given a confidence level $1-\alpha\in[0,1]$, we are interested in one-sided confidence intervals $[\hat q, 1]$ on $\qb$, i.e.~statistics $\hat q_\alpha(T_1,\ldots,T_n)$ such that~\cite{Fisher35, Neyman37}
\begin{equation}\label{eq:confintDef}
P(\qb \geq \hat q) \geq 1-\alpha.
\end{equation}
In particular, we focus on the case where the statistics $\hat q=\hat q_\alpha(n\Tb, n)$ only depends on the confidence parameter $\alpha$, the number of rounds $n$, and the total number of successes
\begin{equation}
n\Tb = \sum_{i=1}^n T_i.
\end{equation}
The random variable $n\Tb$ follows a Poisson binomial distribution, which has a number of well known properties, see~\cite{Tang19} for a recent review.

When considering random variables that are independent and identically distributed (i.i.d.), the central limit theorem~\cite{Esseen58} allows to construct confidence intervals in a generic way and with little information on the underlying distribution. Since each parameter $q_i$ can be different here, we do not wish to make the i.i.d.~assumption. Nevertheless, generic statistical techniques still allow for the construction of valid confidence intervals in this case. Non-i.i.d.~generalizations of Esseen's result constitute one possibility~\cite{Goldstein10}. Another well known option is given by Hoeffding's 1963 inequality $P(\Tb-\qb \geq t) \leq \exp(-2nt^2)$~\cite{Hoeffding63}, which can be rewitten as Eq.~\eqref{eq:confintDef} with the lower bound
\begin{equation}\label{eq:Hoeffding}
\hat q^\text{H} = \Tb - \sqrt{-\log(\alpha)/2n}.
\end{equation}
While admitting a simple analytical expression, confidence intervals based on generic bounds are typically not optimal. This can be problematic when the amount of available data is limited~\cite{Munich19}. Optimal confidence intervals providing the best possible results within a meaningful class of bounds are then particularly interesting~\cite{Buehler57,Lloyd03}.

In 1974, Agnew used a theorem by Hoeffding (distinct from Hoeffding's 1963 inequality) to construct a confidence interval for the average success probability of independent Bernoulli trials~\cite{Agnew74}. This interval is tighter than Eq.~\eqref{eq:Hoeffding}, but still not always optimal. It is sometimes understood that exact and optimal confidence intervals can be formulated, but may be computationally intractable. Recently, a Buehler-optimal confidence interval for the average success probability of independent Bernoulli trials was described in~\cite{Mattner15}. This work also provides a confidence interval which takes a simple form, but it is not applicable to arbitrary values of the confidence parameter $\alpha$ (see Remark 1.4 therein). Here, using Hoeffding's characterization for the cumulative distribution of the number of successes with independent trials already employed by Agnew~\cite{Hoeffding56}, we provide a compact formula for the optimal one-sided confidence interval in terms of standard analytical functions for any confidence parameter $\alpha\in[0,1]$. We then also present a simple expression for a confidence interval that is valid for all $\alpha$ and tight for high enough confidence levels.

More precisely, our main result given in Theorem~\ref{mainResult} is an expression of the monotonically tight one-sided confidence interval on $\qb$ which only depends on the confidence parameter $\alpha$, the total number of rounds $n$, and the number of Bernoulli games won $n\Tb$. When $\alpha\leq1/4$, our bound is given by Corollary~\ref{simpleBound} as the simple formula:

\begin{equation}\label{eq:qhatg}
\hat q^g_\alpha(n\Tb,n) =
\begin{cases}
0 & n\Tb = 0\\
\Tb - \frac{1-\alpha}{n(1-\alpha^*(n\Tb,n))}& n\Tb > 0,\ \alpha^*(n\Tb,n)\leq\alpha\leq 1\\
I^{-1}_\alpha(n\Tb,n(1-\Tb)+1) & n\Tb > 0,\ 0\leq\alpha\leq\alpha^*(n\Tb,n),
\end{cases}
\end{equation}
where
\begin{equation}\label{def:alphaStar}
\alpha^*(n\Tb,n)= I_{(n\Tb-1)/n}(n\Tb,n(1-\Tb)+1).
\end{equation}
Here, $I_x(a,b)$ is a standard mathematical function, namely the regularized incomplete Beta function~\cite{Wikipedia} (see also Appendix~\ref{sec:AppendixBetaFunction}), and $I_\alpha^{-1}(a,b)$ is its inverse. We show that the bound given by Eq.~\eqref{eq:qhatg} remains valid when $\alpha>\frac{1}{4}$, allowing its use for all confidence levels.


\section{Inverse of the cumulative distribution function bound}
In 1956, Hoeffding presented a theorem characterizing the cumulative distribution of the random variable $n\Tb$~\cite{Hoeffding56}. The lower-bound on this cumulative distribution is given by
\begin{equation}\label{eq:cumulative}
P(n\Tb \leq d) \geq 1 - f_{\qb}(d, n)
\end{equation}
with
\begin{equation}\label{eq:deff}
f_{\qb}(d, n) =
\begin{cases}
1                 & d \leq n\qb-1\\
\Qm_n(n-d-1, 1-\qb) & n\qb-1 < d < n\qb\\
I_\qb(d+1, n-d)   & n\qb \leq d.
\end{cases}
\end{equation}
In general, the relevant value of $d$ is in $0,\ldots,n$, but as $f_{\qb}(0,n)$ is not always equal to one, it will be useful later to also consider the above function in the case $d=-1$, with $f_{\qb}(-1, n)=1$ by construction. Therefore we assume that $n\in\mathbb{N}^*$ and $d=-1,\ldots,n$. We also define $c=n-d-1=-1,\ldots,n$ throughout this manuscript. Here, $\Qm_n(c,p)$ is the result of maximizing
\begin{equation}\label{eq:defQs}
\Qf_n(c,p,s) = 1 - I_{\frac{np-s}{n-s}}(c-s+1, n-c)
\end{equation}
over integer values of $s$:
\begin{equation}\label{eq:defQ}
\Qm_n(c,p) = \underset{0\leq s\leq c, s\in\mathbb{Z}}{\max} \Qf_n(c,p,s).
\end{equation}
Note that the function $\Qm_n(c,p)$ is used in Eq.~\eqref{eq:deff} only when $n\qb-1<d$, i.e.~for values of $c$ and $d$ in the range $0,\ldots, n-1$. 

As shown in~\cite{Hoeffding56}, we emphasize that Eq.~\eqref{eq:cumulative} is tight. When $d\leq n\qb-1$, the bound is reached by setting $q_i=1$ for $i\leq d+1$ and $q_i=\frac{n\qb-d-1}{n-d-1}$ for $i=d+2,\ldots,n$. In the remaining cases, the bound is reached by setting $q_i=0$ for $i\leq s$ and $q_i=\frac{n\qb}{n-s}$ for $i>s$, for various choices of $s$. For any $s$ such an assignment leads to 
\be\label{eq: cumul s}
P(n\bar T \leq d) = \sum_{k=0}^d \binom{n-s}{j} \left(\frac{n\qb}{n-s}\right)^k \left(1-\frac{n\qb}{n-s}\right)^{n-s-k} \!\!\!= I_{\frac{n (1-\qb) -s}{n-s}}(c-s+1, n-c)= 1- \Qf_n(c,1-\qb,s),
\ee
where we used Eq.~\eqref{eq: beta to binomial}. When $n\qb\leq d$, the bound is reached with $s=0$, i.e.~with the i.i.d.~model $q_i=\qb$. When $n\qb-1 \leq d \leq n\qb$, it is reached for the value of $s$ in $0,\ldots,n-d-1$ that maximizes Eq.~\eqref{eq:defQ}.

The bound Eq.~\eqref{eq:cumulative} on the cumulative distribution of a Poisson binomial distribution is also continuous and increasing in the average parameter $\qb$.

\begin{lemma}\label{thm:fcont}
Let $n\in\mathbb{N}^*$, $-1\leq d\leq n-1$, $\qb\in[0,1]$. The function $f_{\qb}(d,n)$ is continuous on $\qb\in[0,1]$. Moreover, it is strictly increasing on $\qb\in[0,\frac{d+1}{n}]$ and increasing in $\qb\in[0,1]$.
\end{lemma}
\begin{proof}
By definition
\begin{equation}
 \Qf_n(n-d-1, 1-\qb, 0)= 1- I_{1-\qb}(n-d,d+1)= I_{\qb}(d+1,n-d),
\end{equation}
we must thus have
\begin{eqnarray}\label{eq:maxscase3}
I_{\qb}(d+1,n-d)\ &&=\underset{0\leq s\leq n-d-1, s\in\mathbb{Z}}\max \Qf_n(n-d-1, 1-\qb, s)\\
&&=\Qm_n(n-d-1,1-\qb)
\end{eqnarray}
when $n\qb\leq d$. Otherwise, assigning $q_i=0$ for $i\leq s$ and $q_i=\frac{n\qb}{n-s}$ for $i>s$  would yield 
$P(n\bar T \leq d ) = 1 - \Qf_n(n-d-1,1-\qb,s) < I_{\qb}(d+1,n-d)$ for some $s>0$ (by Eq.~\eqref{eq: cumul s}), and violate the bound Eq.~\eqref{eq:cumulative}. Therefore, the function $f_\qb(d,n)$ can be expressed more concisely as
\begin{equation}\label{eq:f2}
f_{\qb}(d, n) =
\begin{cases}
1                 & d \leq n\qb-1\\
\Qm_n(n-d-1, 1-\qb) & n\qb-1 < d.
\end{cases}
\end{equation}
This function is obviously constant, hence continuous, in its first branch, which includes the case $d=-1$. It is also continuous at $\qb=\frac{d+1}{n}$: by construction, the value of $\Qm_n(n-d-1,1-\qb)$ cannot be larger than 1, and this value is reached in Eq.~\eqref{eq:defQs} with the choice $s=n-d-1$. Therefore, $\Qm_n(n-d-1,1-\qb)=1$ for $\qb=\frac{d+1}n$.

Finally, to show the continuity and strict monotonicity in the domain $\qb\in[0,\frac{d+1}{n}]$ (with $d\geq 0$), we note that for $0\leq s\leq n-d-1$, the parameters of the Beta function in Eq.~\eqref{eq:defQs} satisfy $c-s+1=n-d-s\geq 1$, $n-c=d-1\geq 1$ and $\frac{np-s}{n-s}\in[0,1]$. Since $I_x(a,b)$ with $a,b>0$ and $x\in[0,1]$ is the cumulative distribution of the Beta distribution $B(a,b)$, which is a strictly positive function, the function $\Qf_n(n-d-1,1-\qb,s)$ is continuous and strictly increasing with $\qb$ for all $0\leq s\leq n-d-1$. $\Qm_n(n-d-1,1-\qb)$ is then the maximum over a finite set of such $\Qf_n(n-d-1,1-\qb,s)$ functions, so $\Qm_n(n-d-1,1-\qb)$ is also strictly increasing with $\qb$.
\end{proof}

We are interested in the inverse function of $f_\qb(d,n)$ over $\qb$. When $\qb \leq d/n$, the function $f_\qb(d,n)$ is given by $I_\qb(d+1,n-d)$. Its inverse on the image interval $\alpha=I_\qb(d+1,n-d)\in[0,\alpha^\dag]$ with
\begin{equation}
\alpha^\dag=I_{d/n}(d+1,n-d)
\end{equation}
is then simply given by
\begin{equation}\label{eq:Iinv}
I_\alpha^{-1}(d+1,n-d).
\end{equation}

For compactness, we do not write the functional dependence of $\alpha^\dag$ explicitly, but note that it is a function of $d$ and $n$. When expressed as a function of $c$ and $n$, it takes the form
\begin{equation}\label{eq:alpha_dag_c}
\alpha^\dag=1-I_{(c+1)/n}(c+1,n-c).
\end{equation}

When $d/n<\qb<(d+1)/n$, $f_\qb(d,n)$ is defined in terms of the function $\Qm_n(c,p)$. Due to lemma~\ref{thm:fcont}, this function is bijective. The following lemma characterizes its inverse on $p$.
$\Tf$
\begin{lemma}\label{thm:Qinv}
Let $n\in\mathbb{N}^*$ and $0\leq c\leq n-1$. The inverse of the function $\Qm_n(c,p)$ in the domain $p\in[\frac{c}{n}, \frac{c+1}{n}] \to \alpha\in[\alpha^\dag, 1]$, is given by
\begin{equation}\label{eq:defT}
\Tm_n(c,\alpha)=\underset{0\leq s\leq c, s\in\mathbb{Z}}{\max} \Tf_n(c,\alpha,s),
\end{equation}
where
\begin{equation}
\Tf_n(c,\alpha,s)=\frac{s + (n-s)I_{1-\alpha}^{-1}(c-s+1,n-c)}{n}.
\end{equation}
\end{lemma}

\begin{proof}
By the continuity of $f_{\qb}(d,n)$ shown in Lemma~\ref{thm:fcont}, $\Qm_n(c,\frac{c}{n})=1$ and $\Qm_n(c,\frac{c+1}{n})=\alpha^\dag$. The inverse of $\Qm_n(c,p)$ is then the function $\Qm_n^{-1}(c,\alpha)$ defined from $\alpha\in[\alpha^\dag, 1] \to p\in[\frac{c}{n}, \frac{c+1}{n}]$ which satisfies $\Qm^{-1}_n(c,\Qm_n(c,p))=p$ for all $p\in[\frac{c}{n},\frac{c+1}{n}]$.

Before checking that $\Tm_n(c,\alpha)$ satisfies this condition, we notice that the function $\Tf_n(c,\alpha,s)$ is, by construction, the inverse of $\Qf_n(c,p,s)$ over a larger domain: it satisfies the conditions
\begin{eqnarray}
\Tf_n(c,\Qf_n(c,p,s),s)=p \label{eq:TQ} \\
\Qf_n(c, \Tf_n(c,\alpha,s) ,s)=\alpha \label{eq:QT}
\end{eqnarray}
for all $n\in\mathbb{N}^*,0\leq c\leq n-1,0\leq s\leq c$, and in the domain $p\in[\frac{s}{n},1]$, $\alpha\in[0,1]$.

Let us now show that the value $s'$ of $s$ which maximizes $\Qf_n(c,p,s)$ in Eq.~\eqref{eq:defQ} for a given value of $n$, $c$ and $p$, also maximizes $\Tf_n(c,\alpha',s)$ in Eq.~\eqref{eq:defT} for the same values of $n$ and $c$, and for $\alpha'=\Qf_n(c,p,s')\in[0,1]$. Since
\begin{equation}
\Qf_n(c,p,s')=\underset{0\leq s\leq c, s\in\mathbb{Z}}{\max} \Qf_n(c,p,s),
\end{equation}
we can write for all $0\leq s''\leq c$:
\begin{align}
\Qf_n(c,p,s'') & \leq \Qf_n(c,p,s')\\
& = \alpha'\\
&= \Qf_n(c, \Tf_n(c,\alpha',s''),s''))\\
&= \Qf_n(c, p'', s''),
\end{align}
where $p''=\Tf_n(c,\alpha',s'')$ and we used Eq.~\eqref{eq:QT}. Since the function $\Qf_n(c,p,s'')$ is monotonously decreasing in $p$, we obtain
\begin{equation}
\Tf_n(c,\alpha',s'') = p'' \leq p = \Tf_n(c,\alpha',s')\ \ \ \forall\ 0\leq s''\leq c,
\end{equation}
meaning that $\Tf_n(c,\alpha',s)$ is maximized when $s$ is set to $s'$:
\begin{equation}\label{eq:TsPrime}
\Tm_n(c,\alpha')= \underset{0\leq s\leq c, s\in\mathbb{Z}}{\max} \Tf_n(c,\alpha,s) = \Tf_n(c,\alpha',s').
\end{equation}

Therefore, we find that $\Tm_n(c,\alpha')$ for $\alpha'=\Qm_n(c,p)$ can be written
\begin{eqnarray}
\Tm_n(c,\Qm_n(c,p)) \ &&= \Tm_n(c,\alpha')\\
&&= \Tf_n(c,\alpha',s')\\
&&=\Tf_n(c,\Qf_n(c,p,s'),s')\\
&&=p,
\end{eqnarray}
where we used Eq.~\eqref{eq:TsPrime} and the identity Eq.~\eqref{eq:TQ}. We thus have $\Qm_n^{-1}(c,\alpha) = \Tm_n(c,\alpha)$.
\end{proof}

Let us define the function
\begin{equation}\label{eq:defftilde}
\tilde f_\alpha(d,n) =
\begin{cases}
0 & d=-1\\
1-\Tm_n(n-d-1,\alpha)          & d\geq 0,\ \alpha^\dag < \alpha \leq 1\\
I_{\alpha}^{-1}(d+1, n-d)  & d\geq 0,\ 0 \leq \alpha \leq \alpha^\dag
\end{cases}
\end{equation}
for $\alpha\in[0,1]$, $-1\leq d\leq n$, $n\in\mathbb{N}^*$. The following proposition summarizes the discussion of the inverse of the function $f_\qb(d,n)$ with respect to $\qb$ when $d\geq0$. Note that $f_\qb(-1,n)=1$ is constant and therefore the case $d=-1$ does not admit an inverse.

\begin{proposition}\label{prop:Inverse}
Let $n\in\mathbb{N}^*$ and $0\leq d\leq n$. The function $f_\qb(d,n)$ is bijective, strictly increasing on the domain $\qb\in[0,\frac{d+1}{n}] \to \alpha\in[0,1]$. Its inverse on this domain is given by
\begin{equation}\label{eq:deffinverse}
f^{-1}_\alpha(d,n)=\tilde f_\alpha(d,n).
\end{equation}
\end{proposition}

\begin{proof}
This is a direct consequence of lemma~\ref{thm:fcont}, Eq.~\eqref{eq:Iinv} when $\qb\in[0,\frac{d}{n}]$, $\alpha\in[0,\alpha^\dag]$ and Lemma~\ref{thm:Qinv}  when $\qb\in[\frac{d}{n},\frac{d+1}{n}]$, $\alpha\in[\alpha^\dag,1]$.
\end{proof}

\section{One-sided confidence intervals}
Having recalled the form of the function $f_\qb(d,n)$, which bounds the cumulative distribution of the sum of $n$ Bernoulli random variables $n\Tb$, Eq.~\eqref{eq:cumulative}, and found its inverse, we can now proceed to compute a confidence interval for $\qb$. A naive idea here would be to associate to any $\alpha$, $\qb$ and $n$ a value $d^*$ such that $f_\qb (d^*, n) = \alpha$. Eq.~\eqref{eq:cumulative} and the fact that $f_\qb(d,n)$ is a decreasing function of $d$ would imply that $1-\alpha \leq P(n\Tb \leq d^*)\leq P\Big(f_\qb(n\Tb, n) \geq \alpha = f_\qb(d^*, n)\Big)$. We could then obtain a confidence interval on $\qb$ by applying the inverse function $f^{-1}_\alpha(d,n)$ inside the probability. Yet, this argumentation is clearly not directly applicable because the bound \eqref{eq:cumulative} is restricted to integer values $d=-1,\dots, n$. To take care of this fact we now compute the confidence interval in two steps.

\begin{lemma}\label{lem:lemma}
The random variable $f_\qb(n\Tb-1,n)$ satisfies the following inequality:
\begin{equation}
P\Big(f_\qb( n\Tb-1,n)\geq \alpha \Big)\geq 1-\alpha.
\end{equation}
\end{lemma}

\begin{proof}
The cases $\alpha=0,1$ can be verified directly. Let us thus assume that $\alpha\in(0,1)$. Since $f_\qb(d,n)$ is a tight upper bound on $1-P(n \bar T\leq d)$, it is monotonously decreasing in the integer variable $d$: $f_\qb(d,n)\geq f_\qb(d+1,n)$. Moreover, as $f_\qb(-1,n)=1$ and $f_\qb(n,n)=0$ for any parameter $\qb\in[0,1]$ and $\alpha\in(0,1)$, there exists an integer $-1\leq d^*\leq n-1$ such that 
\be
f_\qb(d^*,n) \geq \alpha \geq f_\qb(d^*+1,n).
\ee
Then, using the fact that
\begin{align}
X \leq Y \Rightarrow f_\qb(X,n) \geq f_\qb(Y,n)
\end{align}
for random variable $X$ and $Y$ implies
\begin{align}\label{eq:ineqforP} 
P(X \leq Y) \leq P(f_\qb(X,n) \geq f_\qb(Y,n)),
\end{align}
we have the following chain of inequalities:
\begin{align}
P\Big(f_\qb( n \Tb-1,n)\geq \alpha \Big) &\geq P\Big(f_\qb( n \Tb-1,n)\geq f_\qb(d^*,n)\Big) \\
 & \geq P\Big(n \Tb-1 \leq d^*\Big)\\ 
 & = P\Big(n \Tb \leq d^*+1 \Big).
\end{align}
By Eq.~\eqref{eq:cumulative} we then have 
\begin{align}
P\Big(f_\qb( n \Tb-1,n)\geq \alpha \Big) & \geq 1- f_\qb(d^*+1,n)\\
 & \geq 1- \alpha.
\end{align}
\end{proof}

We can now show that $[\hat q^f, 1]$ with
\begin{equation}
\hat q^{f}_\alpha(n\Tb,n) = \tilde f_\alpha(n\Tb-1,n)
\end{equation}
is a confidence interval on $\qb$.

\begin{theorem}\label{mainResult}
The two following confidence intervals on the average winning probability $\qb$ hold:
\begin{eqnarray}
P(\qb &\geq& \hat q^f_\alpha(n\Tb,n)) \geq 1-\alpha \label{eq:confInt1}\\
P(\qb &\leq& 1-\hat q^f_\alpha(n(1-\Tb),n)) \geq 1-\alpha \label{eq:confInt2}
\end{eqnarray}
\end{theorem}
\begin{proof}
The second inequality, Eq.~\eqref{eq:confInt2}, is obtained from the first one by performing the change of variable $q_i \to 1-q_i$, $\qb \to 1-\qb$, $\Tb \to 1-\Tb$. We thus focus on the first inequality, Eq.~\eqref{eq:confInt1}. The case $n\Tb=0$ can be checked explicitely, therefore we assume $n\Tb\geq 1$. In this case, $\tilde f_\alpha(n\Tb-1,n) = f^{-1}_\alpha(n\Tb-1,n)$ and so we need to show that $P(\qb \geq f^{-1}_\alpha(n\Tb-1,n)) \geq 1-\alpha$.

By Lemma~\ref{thm:fcont}, the function $f_\qb(d,n)$ is strictly increasing on $\qb=[0,\frac{d+1}{n}]$, with inverse $f_\alpha^{-1}(d,n)$, and simply increasing on $\qb=[0,1]$. We then have the inequality
\begin{equation}
\tilde q=f^{-1}_{f_\qb(d,n)}(d,n) \leq \qb,
\end{equation}
with equality when $\qb\in[0,\frac{d+1}{n}]$. For a random variable $X$, this implies
\begin{equation}
P(\qb \geq X) \geq P(\tilde q \geq X).
\end{equation}
With $d=n\Tb-1$ and $X=f^{-1}_\alpha(n\Tb-1,n)$ we find:
\begin{align}
P(\qb \geq f^{-1}_\alpha(n\Tb-1,n)) & \geq P(f^{-1}_{f_\qb(n\Tb-1,n)}(n\Tb-1,n) \geq f^{-1}_\alpha(n\Tb-1,n))\\
&\geq P(f_\qb(n\Tb-1,n)\geq \alpha)\\
&\geq 1-\alpha,
\end{align}
where we used that $f^{-1}_\alpha(d,n)$ is an increasing function of $\alpha$, as well as Lemma~\ref{lem:lemma}.
\end{proof}


\section{Optimality}

Having established a lower bound on $\qb$, we now study its optimality. Since $\mathbb{E}(n\Tb)=n\qb$ monotonically increases with $\qb$, observing a larger value of $n\Tb$ suggests that $\qb$ is larger. Hence, we consider confidence intervals $\hat q_\alpha(n\Tb,n)$ which are increasing in $n\Tb$, i.e.~satisfying
\begin{equation}
\hat q_\alpha(n\Tb,n)\leq \hat q_\alpha(n\Tb+1,n)
\end{equation}
for $0\leq n\Tb\leq n$. Following Buehler~\cite{Buehler57,Lloyd03}, we then say that a bound $\hat q_\alpha(n\Tb,n)$ is optimal if
\begin{equation}
\hat q_\alpha(n\Tb,n)\geq\hat q'_\alpha(n\Tb,n)\ \ \ \forall\ 0\leq n\Tb\leq n
\end{equation}
holds for all other increasing confidence interval bound $\hat q'_\alpha(n\Tb,n)$.

\begin{proposition}
Let $n\in\mathbb{N}^*$, $\alpha<1$. The statistic $\hat q_\alpha^f(n\Tb,n)$ is the largest monotonically increasing function of $n\Tb$ that satisfies the condition
\begin{equation}
P(\qb\geq\hat q)\geq 1-\alpha
\end{equation}
for all sets of independent Bernoulli random variables $T_1$, $T_2$, \ldots, $T_n$.
\end{proposition}
\begin{proof}
Since $q_i=0\ \forall i$ produces $n\Tb=0$ with $\qb=0$, it is clear that $\hat q_\alpha(0,n)$ must be equal to zero whenever $\alpha<1$ and therefore $\hat q_\alpha(n\Tb,n)=\hat q^f_\alpha(n\Tb,n)$ is optimal in this case. We thus restrict our attention now to the case $n\Tb\geq 1$, for which $\tilde f_\alpha(n\Tb-1,n)=f^{-1}_\alpha(n\Tb-1,n)$.

To see that $\hat q^f_\alpha(n\Tb,n)$ constitutes the best lower bound when $n\Tb\geq1$, consider any increasing confidence interval $\hat q^?_\alpha(n\Tb,n)$ on $\qb$ which associates to a point $n\Tb=d^*+1$ some value $y$, and let us show that we must have $y\leq f^{-1}_\alpha(d^*,n)$. This guarantees that $\hat q^?_\alpha(n\Tb,n) \leq \hat q_\alpha^f(n\Tb,n)$ at any value $d^*\in[1,n]$.

To do so let us introduce the smallest increasing statistic $\hat q'_\alpha(n\Tb,n)$, that satisfy $\hat q'_\alpha(d^*+1,n) = y$. It has the shape of a step function:
\begin{equation}\label{eq:qprime}
\hat q'_\alpha(n\Tb,n) =
\begin{cases}
0 & n\Tb < d^*+1\\
y & n\Tb \geq d^*+1,
\end{cases}
\end{equation}
and satisfies $\hat q'_\alpha(n\Tb,n)\leq \hat q^?_\alpha(n\Tb,n)$, and therefore $P(\qb\geq \hat q'_\alpha(n\Tb,n))\geq P(\qb\geq \hat q^?_\alpha(n\Tb,n))\geq 1-\alpha$, by construction: since $[q^?_\alpha(n\Tb,n),1]$ is a confidence interval, so is $[q'_\alpha(n\Tb,n),1]$. In other words, if some choice of $y$ prevents $[q'_\alpha(n\Tb,n),1]$ from being a confidence interval, this choice of $y$ is not admissible to $q^?_\alpha(n\Tb,n)$ either. So it is sufficient to prove that we must have $y\leq f^{-1}_\alpha(d^*,n)$ for $[\hat q'_\alpha(n\Tb,b),1]$ to be a valid confidence interval.

For a model with some parameter $\qb$, the probability that the estimator $\hat q'$ gives a valid lower bound on $\qb$ is given by
\begin{equation}
P(\qb\geq \hat q'_\alpha(n\Tb,n)) =
\begin{cases}
1 & \qb \geq y\\
P(n\Tb<d^*+1) & \qb < y.
\end{cases}
\end{equation}
In particular, for a model with average winning probability $\qb=y-\delta$, with $\delta>0$, the above probability is given by
\begin{align}
P(\qb\geq \hat q'_\alpha(n\Tb,n)) &= P(n\Tb<d^*+1)\\
&= P(n\Tb \leq d^*)\\
&\geq 1-f_{y-\delta}(d^*,n)
\end{align}
where we used the bound on the cumulative distribution of $n$ independent Bernoulli trials from Eq.~\eqref{eq:cumulative}. In fact, this bound is achievable by a choice of parameters $\{q_i\}_i$ satisfying $\sum_i q_i = n(y-\delta)$~\cite{Hoeffding56}. Thus, in this case we have:
\begin{equation}
\begin{split}
P(\qb\geq \hat q'_\alpha(n\Tb,n)) &= 1-f_{y-\delta}(d^*,n).
\end{split}
\end{equation}
Therefore, for $\hat q'_\alpha(n\Tb,n)$ to be a valid lower bound on $\qb$ with confidence parameter $\alpha$, we must have
\begin{equation}
f_{y-\delta}(d^*,n) \leq \alpha
\end{equation}
Since $\alpha<1$, this is only possible if $y-\delta \leq \frac{d^*+1}{n}$. In this case $f_{y-\delta}(d^*,n)$ is bijective and its inverse is given by Proposition~\ref{prop:Inverse}. Hence we find
\begin{equation}
y -\delta \leq f^{-1}_\alpha(d^*,n).
\end{equation}
Since this condition has to be true for all $\delta>0$, we find the maximum possible value of $y$ in Eq.~\eqref{eq:qprime}: $y \leq f^{-1}_\alpha(d^*,n)$.
\end{proof}

Note that a similar argument implies that Eq.~\eqref{eq:confInt2} is also a Buehler optimal upper bound on $\qb$.

Another consequence from this proposition is that the statistic $\hat q^{0}_\alpha(n\Tb,n)=I^{-1}_\alpha(n\Tb, n-n\Tb+1)$, sometimes referred to as the Clopper-Pearson bound, does not provide a valid confidence interval for $\qb$, as already noticed earlier~\cite{Mattner15}. This statistic would yield the optimal confidence interval in presence of i.i.d.~Bernoulli variables, i.e.~for a binomial trial, but it is not valid here when the trials need not be identical: for the value $n\Tb=1$, $\hat q^{0}_\alpha(1,n)=I^{-1}_\alpha(1,n)= 1-(1-\alpha)^{1/n} > \hat q_\alpha(1,n)=\tilde f_\alpha(n\Tb-1,n)=\frac{\alpha}{n}$. The tightest binomial-like confidence interval is thus given by the statistic
\begin{equation}\label{eq:MunichBound}
\hat q^{1}_\alpha(n\Tb,n)=
\begin{cases}
0 & n\Tb \leq 1\\
I^{-1}_\alpha(n\Tb-1,n-n\Tb+2) & n\Tb \geq 2,
\end{cases}
\end{equation}
shown to yield a valid confidence interval when $\alpha \leq 1/2$ in~\cite{Munich19}. Since $\lim_{\alpha\to1}I^{-1}_\alpha(n\Tb-1,n-n\Tb+2)=1$ for $n\Tb\geq 2$, it is clear that this bound cannot be true for all confidence parameters $\alpha$, contrary to $\hat q^f_\alpha(n\Tb,n)$.

\section{A simpler explicit bound}
The expression for the confidence interval given in Theorem~\ref{mainResult} involves a function maximization through the definition of the $\Tm_n(c,\alpha)$ function. This maximization requires at most $c+1\leq n+1$ evaluations of the function $\Tf_n(c,\alpha,s)$, which scales favorably with $n$. Still, a confidence interval which does not involve such maximization would be appreciable. For this, let us show that the functions $\Qm_n(c,p)$ and $\Tm_n(c,\alpha)$ can be upper bounded with linear functions of $p$ and $\alpha$ respectively.

\begin{lemma}\label{linearBounds}
For $\frac{c}{n}\leq p \leq \frac{c+1}{n}$, $0 \leq c \leq n-1$, $\alpha^\dag\leq \alpha\leq 1$, we have the following two inequalities:
\begin{eqnarray}
\Qm_n(c,p) &\leq& R_n(c,p) = 1 - (1-\alpha^\dag)(np-c) \label{eq:Qupper}\\
\Tm_n(c,\alpha) &\leq& U_n(c,\alpha)=\frac{1}{n}\left(c + \frac{1-\alpha}{1-\alpha^\dag}\right) \label{eq:Tupper}
\end{eqnarray}
Moreover, these inequalities are tight for $c=n-1$.
\end{lemma}
\begin{proof}
We start by showing the bound Eq.~\eqref{eq:Qupper}. As $\Qm_n(c,p) = \max_s \Qf_n(c,p,s)$, it is implied by
\begin{equation}\label{eq: 69}
    \Qf_n(c,p,s) = 1 - I_{\frac{np-s}{n-s}}(c-s+1, n-c)\leq 1 - (1-\alpha^\dag)(np-c)
\end{equation}
for all integer $0\leq s\leq c$. Let us introduce the functions  
\begin{align}\label{eq: x(p)}
    x(p) &=\frac{n p-s}{n-s}= 1 -\frac{n}{n-s}(1-p) \in \left[0,\frac{c}{n}\right]\\
    \label{eq: ell(p)}
    \ell(p)&=(1-\alpha^\dag)(np-c) \in [0,1-\alpha^\dag],
\end{align}
where we recalled that $\alpha^\dag$ can be expressed as a function of $c$ through Eq.~\eqref{eq:alpha_dag_c}, and rewrite the bound Eq.~\eqref{eq: 69} as
\begin{equation}\label{eq:toprove}
     I_{x(p)}(a, b) \geq \ell(p),
\end{equation}
where $a = c-s +1\geq 1$, $b=n-c \geq 1$. To demonstrate it we proceed by considering several cases.\\

For $b=1$, i.e.~$c=n-1$ and $a= n-s$, one gets $\alpha^\dag = 0$ and the inequality~\eqref{eq:toprove} becomes
\begin{eqnarray}\label{eq: temp 1}
 &&I_{x(p)}(n-s, 1) \geq np-c
    \\ \label{eq: 72}
    && \Longleftrightarrow \left(1- \frac{n}{n-s}(1-p)\right)^{n-s} \geq 1- n(1-p) \\
    &&\Longleftrightarrow \left(1- \frac{1-\delta}{n-s}\right)^{n-s} \geq \delta.
\end{eqnarray}
In the last line we introduced the variable $\delta = 1-n(1-p)$, which satisfies $0\leq \delta\leq 1$ as guaranteed by $\frac{c}{n}\leq p \leq \frac{c+1}{n}$ and $c= n-1$. We note that the inequality is manifestly tight for $s=n-1$ (it becomes $\delta\geq \delta$), which shows that Eq.~\eqref{eq:Qupper} is tight for $c=n-1$ as claimed. To prove that the bound holds in general, i.e.~for $s\leq n-1$, we now show that the left hand side is a monotonically decreasing function of $s\in [0,n-1]$. Equivalently, the function
\begin{equation}
 f(v)=    \left(1- \frac{1-\delta}{v}\right)^v
 \end{equation}
 $1 \leq v = n-s \leq n$ is monotonically increasing in $v$. This can be done by computing its derivative, which satisfies 
\begin{equation}
    \frac{\partial}{\partial v} \left(1- \frac{1-\delta}{v}\right)^v \propto \frac{1-z}{z} + \log(z) \geq 0
\end{equation}
for $0< z=1-\frac{1-\delta}{v}\leq 1$, using the standard lower bound on the natural logarithm $1-\frac{1}{z} \leq \log(z)$ for $z>0$. In the limiting case $z=0$, corresponding to $p=1$, the inequality~\eqref{eq: 72} holds trivially.

Next we consider the case $b\geq 2$. As a preliminary step we analyze the concavity/convexity of $I_{x(p)}(a, b)$ as a function of $p$ in the interval $p\in \left[\frac{c}{n}, \frac{c+1}{n} \right]$. To do so let us compute the derivatives of $I_{x(p)}(a, b)$ with respect to $p$, which can be done straightforwardly with the help of the integral representation of the incomplete Beta function
\begin{equation}
     I_{x}(a,b) = a \binom{a+b-1}{b-1}  \int_0^x t^{a-1} (1-t)^{b-1} d t.
\end{equation}
For the first derivative we obtain
\begin{eqnarray}\label{eq: I'}
     I_{x}'(a,b)\ &&=  \frac{\partial}{\partial p} I_{x(p)}(a,b) = \frac{\partial x(p)}{\partial p} \frac{\partial I_{x}(a,b)}{\partial x} \\
     &&= a \binom{a+b-1}{b-1}   x'(p)\,  \frac{\partial}{\partial x} \int_0^x  t^{a-1} (1-t)^{b-1} d t 
     \\
     && = a \binom{a+b-1}{b-1} \frac{n}{n-s}\,  x(p)^{a-1} \big(1-x(p)\big)^{b-1}.
\end{eqnarray}
Computing the second derivative is also straightforward and gives
\begin{eqnarray}
     I_{x}''(a,b)\ &&=  \frac{\partial^2}{\partial p^2} I_{x(p)}(a,b) \\
     && = a \binom{a+b-1}{b-1} \frac{n}{n-s}\,\frac{\partial}{\partial p}  x(p)^{a-1} \big(1-x(p)\big)^{b-1}\\
     && = a \binom{a+b-1}{b-1} \left(\frac{n}{n-s}\right)^2
     \begin{cases}
     x^{a-2}(1-x)^{b-2} \Big(a-1 - (a+b-2) x \Big) & a\geq 2\\
     -(b-1) (1-x)^{b-2}   & a=1 \\
     \end{cases},
\end{eqnarray}
where we wrote $x= x(p)$ to improve readability. We will now analyze the two regimes separately.\\

For $a=1$, that is $s=c$, the second derivative is negative, hence $I_{x(p)}(1,b)$ is concave. Furthermore, for the minimal value of $p=\frac{c}{n}$  the bound~\eqref{eq:toprove} that we want to prove is saturated
\begin{equation}
    I_{x(\frac{c}{n})}(1,b) = 0 = \ell\left(\frac{c}{n}\right).
\end{equation}
Thus, since $I_{x(p)}(1,b)$ is concave and $\ell(p)$ is linear, to demonstrate the bound~\eqref{eq:toprove} on the whole interval $p\in \left[\frac{c}{n}, \frac{c+1}{n} \right]$ it is sufficient to show that it holds for the maximal value of $p=\frac{c+1}{n}$, i.e.
\begin{equation}
    (i)\qquad I_{x(\frac{c+1}{n})}(1,b) \geq \ell\left(\frac{c+1}{n}\right).
\end{equation}
We show (i) below.\\

For $a\geq 2$ the sign $I_{x(p)}''(a,b)$ can be determined from 
\be
I_{x(p)}''(a,b) \propto a-1 - (a+b-2)x(p) = a-1-\frac{(a+b-2) (n-c p)}{n-c}.
\ee
It is easy to see that $I_{x(p)}''(a,b) \geq 0 $ for $\frac{c}{n}\leq p \leq p_*= \frac{c n-c s-s}{n (n-s-1)}$ and $I_{x(p)}''(a,b) < 0 $ for $p_*< p\leq \frac{c+1}{n}$. In particular, this shows that the derivative of $I_{x(p)(a,b)}$ attains its global maximum at $p_*$:
\begin{equation}\label{eq:global I'}
    I'_{x(p)}(a,b) \leq I'_{x(p_*)}(a,b)
\end{equation}\\
for $\frac{c}{n}\leq p \leq \frac{c+1}{n}$. Now we can prove that $I_{x(p)}(a,b)\geq \ell(p)$ by showing the following inequalities
\begin{eqnarray}
(ii)\qquad &&I_{x(\frac{c+1}{n})}(a,b) \geq \ell\left(\frac{c+1}{n}\right)\\
(iii) \qquad  &&I'_{x(p)}(a,b) \leq \ell'(p) \quad \text{for} \qquad p\in \left[\frac{c}{n}, \frac{c+1}{n} \right].
\end{eqnarray}
Since $I_{x(p)}(a,b)$ is above $\ell(p)$ at the end of the interval  $p=\frac{c+1}{n}$ (by $(ii)$) and $I_{x(p)}(a,b)$ decrease slower than $\ell(p)$ for $p\leq\frac{c+1}{n}$ (by $(iii)$), Eq.~\eqref{eq:toprove} holds on the whole interval. We now show $(i)$,$(ii)$ and $(iii)$\\

Expressions $(i)$ and $(ii)$ can be combined in a single bound $ 1-I_{x(\frac{c+1}{n})}(a,b) \leq 1-\ell\left(\frac{c+1}{n}\right)$ for $a\geq 1$. Using the Eqs.~(\ref{eq: 69},\ref{eq: x(p)},\ref{eq: ell(p)}) and $c=n-d-1$ we rewrite both sides of this inequality as
\be
1-I_{x(\frac{c+1}{n})}(a,b) = 1-I_{x(\frac{c+1}{n})}(c-s+1,n-c)
= \Qf_n\left(c,\frac{c+1}{n},s\right) = \Qf_n\left(n-d-1,1-\frac{d}{n},s\right)
\ee
and
\be   
1-\ell\left(\frac{c+1}{n}\right) = 1- (1-\alpha^\dag)
= I_{\frac{n-c-1}{n}}(n-c,c+1)
 = I_{d/n}(d+1,n-d)
\ee
for $s\leq n-d-1$. But
\begin{equation}
\Qf_n\left(n-d-1,1-\frac{d}{n},s\right) \leq I_{d/n}(d+1,n-d)
\end{equation}
is implied by Eq.~\eqref{eq:maxscase3} with the choice of $\bar q=\frac{d}{n}$, proving $(i)$ and $(ii)$.\\

To show Eq.~\eqref{eq:Qupper} it thus remains to prove $(iii)$ for $a\geq 2$ (equivalently $s<c$). Since $I'_{x(p)}(a,b)$ is maximized at $p_*$, by Eq.~\eqref{eq:global I'}, and $\ell'(p)$ is constant it is sufficient to verify $(iii)$ at a single point
\begin{equation}\label{eq:toShow}
     I'_{x(p_*)}(a,b) \leq \ell'(p),
\end{equation}
that we need to prove for $s< c < n-1$, i.e.~$a\geq 2$ and $b\geq 2$. First, for the right hand side with the help of Eq.~\eqref{eq: beta to binomial}  we obtain
\begin{align}
\ell'(p) =  n(1-\alpha^\dag)&=n(1- I_{1-\frac{c+1}{n}}(n-c,c+1))\\
&= n\left(1 -\sum_{k=0}^c\binom{n}{k}\left(\frac{c+1}{n}\right)^k
\left(1-\frac{c+1}{n}\right)^{n-k} \right) \\
&= n\left(\sum_{k=c+1}^n\binom{n}{k}\left(\frac{c+1}{n}\right)^k
\left(1-\frac{c+1}{n}\right)^{n-k} \right) \\
& \geq n \frac{1}{2}
\end{align}
for the last line we used the fact that for an integer mean $n p =n\frac{c+1}{n}$ the median of the binomial distribution is given by $\lfloor np \rfloor= \lceil np \rceil = c+1$, see Corollary 1 in \cite{kaas}. Next, with the help of Eq.~\eqref{eq: I'} the left hand side of Eq.~\eqref{eq:toShow} can be written as
\begin{eqnarray}
I'_{x(p_*)}(a,b)\ &&= n \binom{n-s-1}{c-s} \left(\frac{c-s}{n-s-1}\right)^{c-s} \left(1-\frac{c-s}{n-s-1}\right)^{(n-s-1)-(c-s)} \\
&&= n \binom{N}{y} \left(\frac{y}{N} \right)^{y} \left(\frac{N-y}{N} \right)^{N-y}
\end{eqnarray}
with integers $N= n-s-1 > y = c-s \geq 1$. The result of the following Lemma~\ref{maxbinomial} directly implies 
\begin{equation}
    n \binom{N}{y} \left(\frac{y}{N} \right)^{y} \left(\frac{N-y}{N} \right)^{N-y}\leq n \frac{1}{2}.
\end{equation} 
Combining with the bound on the right hand side gives
\begin{equation}
    I'_{x(p_*)}(a,b) \leq \frac{n}{2}\leq \ell'(p)
\end{equation}
and completes the demonstration of Eq.\eqref{eq:Qupper}. \\

We now prove inequality~\eqref{eq:Tupper}. Since $\Qm_n(c,p)$ and $R_n(c,p)$ have identical domains over $p\in[\frac{c}{n},\frac{c+1}{n}]$, for every $\alpha\in[\alpha^\dag,1]$, there exist $p',p''\in[\frac{c}{n},\frac{c+1}{n}]$, such that
\begin{equation}
\alpha = \Qm_n(c,p'') = R_n(c,p').
\end{equation}
Moreover, Eq.~\eqref{eq:Qupper} implies
\begin{eqnarray}
&&\Qm_n(c,p')\leq R_n(c,p') = \Qm_n(c,p'')\\
&&\Longleftrightarrow p' \geq p'',
\end{eqnarray}
because $\Qm_n(c,p)$ is a decreasing function of $p$, c.f.~Lemma~\ref{thm:fcont}.
Since $\Tm$ is the inverse function of $\Qm$, and $U$ of $R$, we obtain as desired
\begin{eqnarray}
\Tm_n(c,\alpha)\ &&= \Tm_n(c,\Qm_n(c,p''))\\
&&= p''\\
&& \leq p'\\
&&=U_n(c,R_n(c,p'))\\
&&=U_n(c,\alpha).
\end{eqnarray}
\end{proof}

\begin{lemma}\label{maxbinomial}
For two integers $N > y > 0$, the following bound holds
\begin{equation}
    \binom{N}{y} \left(\frac{y}{N} \right)^{y} \left(\frac{N-y}{N}\right)^{N-y}  \leq \frac{1}{2}
\end{equation}
\end{lemma}
\begin{proof}
Using the bound on the binomial coefficient
\begin{equation}
\binom{N}{y}\leq\sqrt{\frac{N}{2\pi y(N-y)}}\frac{N^N}{y^y(N-y)^{N-y}}
\end{equation}
from Corollary 1 of~\cite{Sasvari99}, valid when $0<y<N$, we obtain
\begin{equation}
\binom{N}{y} \left(\frac{y}{N} \right)^{y} \left(\frac{N-y}{N} \right)^{N-y}  \leq \sqrt{\frac{N}{2\pi y(N-y)}}.
\end{equation}
Now for $1\leq y\leq N-1$ and $N\geq 3$, this quantity is strictly smaller than $1/2$. Indeed, this function is convex in $y$ and symmetric under exchange $y\to N-y$, so it is upper bounded by its value at $y=1$, which is
\begin{equation}
\sqrt{\frac{N}{2\pi (N-1)}}.
\end{equation}
The derivative of this function in N is negative, so its largest value $\frac{1}{2}\sqrt{\frac{3}{\pi}}<\frac{1}{2}$ is obtained for $N=3$.

It remains to examine the special cases $N<3$. It leaves us with the unique possibility $y=1$ and $N=2$ satisfying $1\leq y=1 \leq N-1 =1$, and
\begin{equation}
\binom{N}{y} \left(\frac{y}{N} \right)^{y} \left(\frac{N-y}{N} \right)^{N-y} = \frac{1}{2}. 
\end{equation}
\end{proof}

We now define the function $g_\qb(d,n)$ as
\begin{equation}
\begin{split}
g_\qb(d,n)=
\begin{cases}
1                 & d \leq n\qb-1\\
1+(1-\alpha^\dag)(n\qb-d-1) & n\qb-1 < d < n\qb\\
I_\qb(d+1, n-d)   & n\qb \leq d
\end{cases}
\end{split}
\end{equation}
and its inverse $g^{-1}_\alpha(d,n)$ as
\begin{equation}\label{eq:defgg}
\begin{split}
g^{-1}_\alpha(d,n)=
\begin{cases}
\frac{1}{n}\left(d+1-\frac{1-\alpha}{1-\alpha^\dag}\right) & \alpha^\dag < \alpha \leq 1\\
I_{\alpha}^{-1}(d+1, n-d)  & 0 \leq \alpha \leq \alpha^\dag,
\end{cases}
\end{split}
\end{equation}
where we recall that $\alpha^\dag = I_{d/n}(d+1,n-d)$. We further extend this function to the value $d=-1$ by defining
\begin{equation}\label{eq:defgtilde}
\tilde g^{-1}_\alpha(d,n)=
\begin{cases}
0 & d = -1\\
\frac{1}{n}\left(d+1-\frac{1-\alpha}{1-\alpha^\dag}\right) & d\geq 0, \alpha^\dag < \alpha \leq 1\\
I_{\alpha}^{-1}(d+1, n-d)  & d \geq 0, 0 \leq \alpha \leq \alpha^\dag,
\end{cases}
\end{equation}

We now show that this allows us to write the simple confidence interval $[\hat q^g, 1]$ on $\qb$ with
\begin{equation}
\hat q^g_\alpha(n\Tb,n) = \tilde g_\alpha^{-1}(n\Tb-1,n).
\end{equation}
For clarity, we refer to $\alpha^\dag$ when $d=n\Tb-1$ as the function $\alpha^*(n\Tb,n)$, as defined in Eq.~\eqref{def:alphaStar}. This allows expressing the statistic $\hat q^g$ directly in terms of $\alpha$, $n\Tb$ and $n$ as given in Eq.~\eqref{eq:qhatg}.

\begin{corollary}\label{simpleBound}
The two following confidence intervals on the average winning probability $\qb$ hold:
\begin{eqnarray}
P(\qb &\geq& \hat q^{g}_\alpha(n\Tb,n)) \geq 1-\alpha \label{eq:confIntSimple1}\\
P(\qb &\leq& 1-\hat q^{g}_\alpha(n(1-\Tb),n)) \geq 1-\alpha \label{eq:confIntSimple2}
\end{eqnarray}
Moreover, these confidence intervals are identical to those given in Prop.~\ref{mainResult} when either $n\Tb\leq 1$ or $\alpha\leq\alpha^*(n\Tb,n)$, and in particular when $\alpha\leq 1/4$.
\end{corollary}

\begin{proof}
The validity of these confidence intervals is a direct consequence of Theorem~\ref{mainResult} and Lemma~\ref{linearBounds}.

It remains to show equality between $\tilde f^{-1}_\alpha(d,n)$ and $\tilde g^{-1}_\alpha(d,n)$ under the stated conditions. First, the case $n\Tb=0$, i.e.~$d=-1$, can be checked explicitly, so we only need to study the equality between $f^{-1}_\alpha(d,n)$ and $g^{-1}_\alpha(d,n)$ when $d \geq 0$. Second, we notice that $\Tm_n(c,\alpha)=U_n(c,\alpha)$ for $c=n-1$. This guarantees that $f^{-1}_\alpha(d,n)=g^{-1}_\alpha(d,n)$ for $d=0$, i.e.~$n\Tb=1$. Third, it is clear from the definition that $f^{-1}_\alpha(d,n)=g^{-1}_\alpha(d,n)$ when $d\geq 0$ and $\alpha\leq\alpha^\dag$. Finally, since equality in the cases $n\Tb\leq 1$ is valid for all $\alpha$, we prove that the condition $\alpha\leq1/4$ is sufficient by showing $\alpha^\dag\geq 1/4$ for $1\leq d\leq n-1$, $n\in\mathbb{N}^*$. With the help of Eq.~\eqref{eq: beta to binomial} we write
\bea
\alpha^\dag & =& I_\frac{d}{n}(d+1,n-d)\\
&=&\sum_{k=d+1}^n\binom{n}{k}\left(\frac{d}{n}\right)^k\left(\frac{n-d}{n}\right)^{n-k}
\eea
and verify $\alpha^\dag\geq 1/4$ in six different regions.

\begin{itemize}
\item In case $d=n-1$, $n\geq 2$, we write 
\begin{equation}
\alpha^\dag =\left(1-\frac{1}{n}\right)^{n}
\end{equation}
This quantity is 1/4 for $n=2$, and increases with $n$. Indeed, its derivative with respect to $n$ is positive:
\begin{equation}
\left(1-\frac{1}{n}\right)^n(n-1)^{-1}\left(1+(n-1)\log\left(1-\frac{1}{n}\right)\right)
\end{equation}
is the product of two positive terms with a third positive one:
\begin{align}
1+(n-1)\log\left(1-\frac{1}{n}\right)&\geq 1+(n-1)\left(1-\frac{n}{n-1}\right) \\
&\geq 0,
\end{align}
where we used the bound $\log(x)\geq 1-\frac{1}{x}$.

\item The case $d=n-2$, $n\geq 3$, can be treated similarly: we write 
\begin{equation}
\alpha^\dag =\left(n-2\right)^{n-1}n^{-n}(3n-2)
\end{equation}
This quantity is larger than 1/4 for $n=3$, and increases with $n$. Indeed, its derivative with respect to $n$ is positive:
\begin{equation}
\left(n-2\right)^{n-2}n^{-n}\left(6n-8 + (4-8n+3n^2)\log\left(1-\frac{2}{n}\right)\right)
\end{equation}
is the product of two positive terms with a third positive one:
\begin{equation}
6n-8 + (4-8n+3n^2)\log\left(1-\frac{2}{n}\right)\geq 2-\frac{4}{n},
\end{equation}
where we used this time the bound $\log(1+x)\geq \frac{x}{2}\frac{2+x}{1+x}$ for $-1<x\leq0$~\cite{Flemming04}.

\item In the case $d=1$, $n\geq 2$ we have
\begin{equation}
\begin{split}
\alpha^\dag&=(n-1)^{n-1}n^{-n}\left(1-2n+(n-1)\left(1-\frac{1}{n}\right)^{-n}\right)
\end{split}
\end{equation}
This quantity is 1/4 for $n=2$, and increases with $n$. Indeed, its derivative with respect to $n$ is positive:
\begin{equation}
\left(n-1\right)^{n-1}n^{-n}\left(-2-(2n-1)\log\left(1-\frac{1}{n}\right)\right)
\end{equation}
is the product of two positive terms with a third positive one:
\begin{equation}
\begin{split}
&-2-(2n-1)\log\left(1-\frac{1}{n}\right) \geq 0,
\end{split}
\end{equation}
where we used the bound $\log(1+x)\leq \frac{2x}{2+x}$ for $-1<x\leq0$~\cite{Flemming04}.

\item In the case $d=2$, $n\geq 3$ we have
\begin{equation}
\begin{split}
\alpha^\dag&=(n-2)^{n-2}n^{-n}\left(-4+10n-5n^2+(n-2)^2\left(1-\frac{2}{n}\right)^{-n}\right)
\end{split}
\end{equation}
This quantity is larger than 1/4 for $n=3$, and increases with $n$. Indeed, its derivative with respect to $n$ is positive:
\begin{equation}
\left(n-2\right)^{n-2}n^{-n}\left(10-10n + (-4+10n-5n^2)\log\left(1-\frac{2}{n}\right)\right)
\end{equation}
is the product of two positive terms with a third positive one:
\begin{equation}
\begin{split}
10-10n + (-4+10n-5n^2)\log\left(1-\frac{2}{n}\right)\geq\frac{4(n-1)}{2-6n+3n^2},
\end{split}
\end{equation}
where we used the bound $\log(1+x)\leq \frac{3x(2+x)}{6+6x+x^2}$ for $-1<x\leq0$~\cite{Flemming04}. This last term is positive when $n\geq1+\frac{1}{\sqrt{3}}\simeq1.58$.

\item All cases with $n\leq16$ and $1\leq d\leq n-1$ can be checked by direct computation, see Table~\ref{table:alphaStar}.

\renewcommand{\arraystretch}{1.3}
\begin{table}
\begin{tabular}{|*{17}{c|}}
\hline
 \diagbox{$n$}{$d$} & 1 & 2 & 3 & 4 & 5 & 6 & 7 & 8 & 9 & 10 & 11 & 12 & 13 & 14 & 15  \\ \hline
16 & 0.264 & 0.323 & 0.352 & 0.370 & 0.382 & 0.391 & 0.397 & 0.402 & 0.405 & 0.407 & 0.407 & 0.405 & 0.400 & 0.388 & 0.356\\ \hline
15 & 0.264 & 0.323 & 0.352 & 0.370 & 0.382 & 0.390 & 0.396 & 0.401 & 0.403 & 0.404 & 0.403 & 0.398 & 0.387 & 0.355\\  \cline{1-15}
14 & 0.264 & 0.323 & 0.352 & 0.369 & 0.381 & 0.389 & 0.395 & 0.399 & 0.401 & 0.400 & 0.396 & 0.385 & 0.354\\ \cline{1-14}
13 & 0.264 & 0.323 & 0.351 & 0.369 & 0.381 & 0.389 & 0.394 & 0.397 & 0.397 & 0.394 & 0.383 & 0.353\\ \cline{1-13}
12 & 0.264 & 0.323 & 0.351 & 0.368 & 0.380 & 0.387 & 0.392 & 0.393 & 0.391 & 0.381 & 0.352\\ \cline{1-12}
11 & 0.264 & 0.322 & 0.351 & 0.368 & 0.379 & 0.385 & 0.388 & 0.387 & 0.379 & 0.350\\ \cline{1-11}
10 & 0.264 & 0.322 & 0.350 & 0.367 & 0.377 & 0.382 & 0.383 & 0.376 & 0.349\\ \cline{1-10}
9 & 0.264 & 0.322 & 0.350 & 0.366 & 0.374 & 0.377 & 0.372 & 0.346\\ \cline{1-9}
8 & 0.264 & 0.321 & 0.349 & 0.363 & 0.370 & 0.367 & 0.344\\ \cline{1-8}
7 & 0.264 & 0.321 & 0.347 & 0.359 & 0.360 & 0.340\\ \cline{1-7}
6 & 0.263 & 0.320 & 0.344 & 0.351 & 0.335\\ \cline{1-6}
5 & 0.263 & 0.317 & 0.337 & 0.328\\ \cline{1-5}
4 & 0.262 & 0.313 & 0.316\\ \cline{1-4}
3 & 0.259 & 0.296\\ \cline{1-3}
2 & 1/4 \\ \cline{1-2}
\cline{1-1}
\end{tabular}
\caption{Numerical values of $\alpha^\dag=I_{d/n}(d+1,n-d)$ for $2\leq n\leq 16$ and $1\leq d\leq n-1$. All values in the table are larger or equal to $1/4$.}
\label{table:alphaStar}
\end{table}

\item To treat the remaing cases where $3\leq d\leq n-2$ and $n \geq 17$, we decompose $\alpha^\dag$ in two terms:
\begin{align}
\alpha^\dag&=\sum_{k=d+1}^n\binom{n}{k}\left(\frac{d}{n}\right)^k\left(\frac{n-d}{n}\right)^{n-k}\\
&=\sum_{k=d}^n\binom{n}{k}\left(\frac{d}{n}\right)^k\left(\frac{n-d}{n}\right)^{n-k}\\
&\ \ -\binom{n}{d}\left(\frac{d}{n}\right)^d\left(\frac{n-d}{n}\right)^{n-d}.\nonumber
\end{align}
Since the median of a binomial distribution with parameter $\frac{d}{n}$ is $d$, the first sum above is larger than $1/2$. We thus need to show that the second term is upper bounded by $1/4$ for the whole remaining parameter region. We first notice that this quantity is indeed smaller than $1/4$ for $n=16$, $3\leq d\leq n-3$, see Table~\ref{table:tab2}. We now show that it is also smaller than $1/4$ when $d=n-3$ and $n\geq 17$.

When $d=n-3$,
\begin{align}
x(n,d)&=\binom{n}{d}\left(\frac{d}{n}\right)^d\left(\frac{n-d}{n}\right)^{n-d}\\
&=\frac{9}{2}(n-3)^{n-3}(n-2)(n-1)n^{-(n-1)}
\end{align}
We show that this quantity decreases with $n$ by upper bounding its derivative:
\begin{equation}
\frac{9}{2}(n-3)^{n-3}n^{-n}\left(2+3n(n-2)+n(n-1)(n-2)\log\left(1-\frac{3}{n}\right)\right)
\end{equation}
is the product of terms which are all positive except the last one:
\begin{equation}
2+3n(n-2)+n(n-1)(n-2)\log\left(1-\frac{3}{n}\right)\leq\frac{6-10n+3n^2}{3-2n},
\end{equation}
where we used the bound $\log(1+x)\leq \frac{2x}{2+x}$ for $-1<x\leq0$~\cite{Flemming04}. The numerator is positive here and the denominator negative, therefore $x(n,n-3)$ decreases with $n$ and is smaller than $1/4$ for all $n\geq 17$.

Finally, we show now that $x(n,d)$ decreases with $n$. The derivative of $x(n,d)$ with respect to $n$ can be written
\begin{equation}
\frac{\partial x(n,d)}{\partial n} = \binom{n}{d}\left(\frac{d}{n}\right)^d\left(1-\frac{d}{n}\right)^{n-d}\left[H(n) - H(n-d) + \log\left(1-\frac{d}{n}\right)\right]
\end{equation}
where $H(n)=\sum_{k=1}^n\frac{1}{k}$ is the harmonic number. The three first terms here are positive. To expand the square bracket, we use the relation~\cite{robjohn17}
\begin{equation}
\log(n)+\gamma+\frac{1}{2n+1} \leq H(n) \leq \log(n)+\gamma+\frac{1}{2n-1},
\end{equation}
where $\gamma$ is the Euler-Mascheroni constant. This last term is then upper bounded by
\begin{equation}
\frac{2(1-d)}{(2n-1)(2(n-d)+1)},
\end{equation}
which is negative when $n\geq 1/2$ and $1\leq d \leq n+1/2$.

\begin{table}
\begin{tabular}{|*{7}{c|}}
\hline
 $d$ & 3, 13 & 4, 12 & 5, 11 & 6, 10 & 7, 9 & 8  \\ \hline
 & &  & & &  &   \vspace{-8pt} \\
 & \scalebox{1.3}{$\frac{286216975729679085}{1152921504606846976}$} & \scalebox{1.3}{$\frac{241805655}{1073741824}$} & \scalebox{1.3}{$\frac{243406518990009375}{1152921504606846976}$} & \scalebox{1.3}{$\frac{7126259765625}{35184372088832}$} & \scalebox{1.3}{$\frac{228126063717356805}{1152921504606846976}$} & \scalebox{1.3}{$\frac{6435}{32768}$} \\
 & $\simeq 0.248254$ & $\simeq 0.225199$ & $\simeq 0.211122$ & $\simeq 0.20254$ & $\simeq 0.197868$ & $\simeq 0.196381$ \\ \hline
\end{tabular}
\caption{Values of $\binom{n}{d}\left(\frac{d}{n}\right)^d\left(\frac{n-d}{n}\right)^{n-d}$ for $n=16$ and $3\leq d\leq n-3$. All values in the table are smaller than $1/4$.}
\label{table:tab2}
\end{table}

\end{itemize}

\end{proof}

\section{Illustrations}
Figure~\ref{fig1} presents a comparison between the bounds provided by Theorem~\ref{mainResult} (Eq.~\eqref{eq:confInt1}) and Corollary~\ref{simpleBound} (Eq.~\eqref{eq:qhatg}), together with Eq.~\eqref{eq:MunichBound} and the bound implied from Hoeffding's 1963 inequality (Eq.~\eqref{eq:Hoeffding}).

\begin{figure}[h]
    \includegraphics[width=\columnwidth]{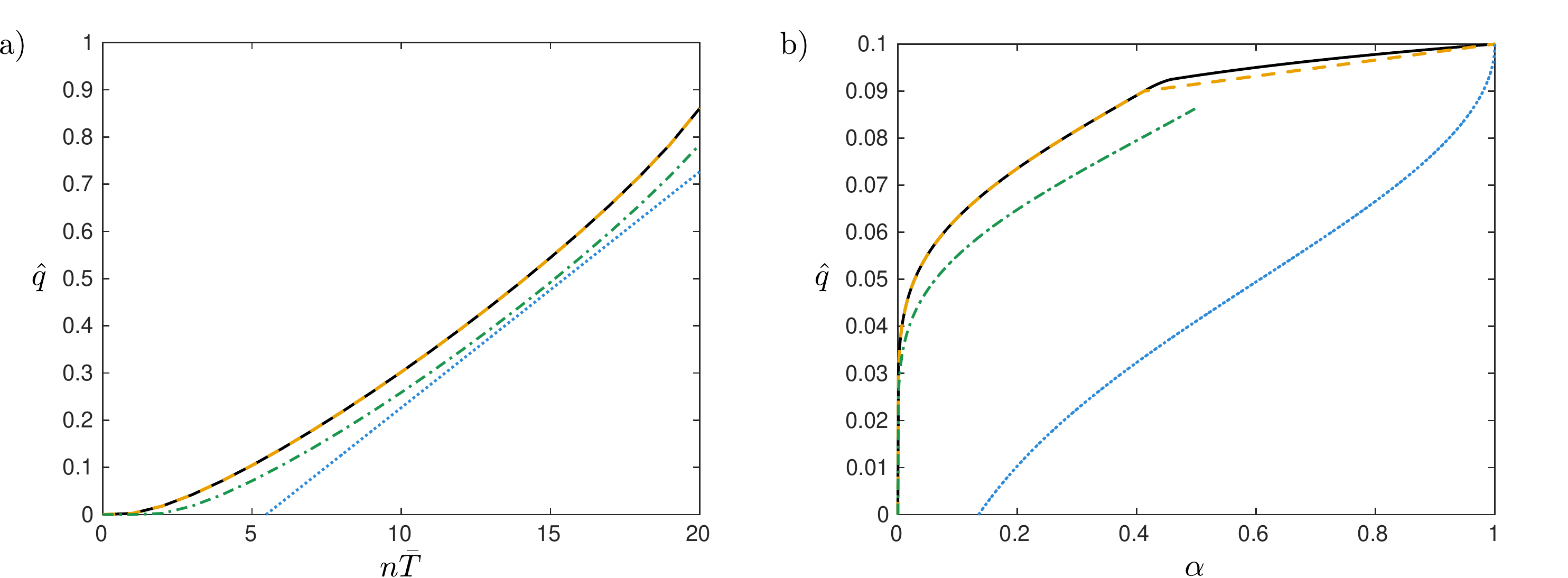}
    \caption{Lower bound $\hat q$ on the average success probability $\qb$ of non-identical Bernoulli trials. The bounds are functions of the total number of sample $n$, the number of successive samples $n\Tb$ and the confidence parameter $\alpha$. Four bounds are represented here as given by $\hat q^f$ (Buehler optimal), $\hat q^g$, $\hat q^1$ and Hoeffding's $\hat q^H$ (from top to bottom).
a) $n=20$ rounds and $\alpha=0.05$. Since $\alpha<1/4$, $\hat q^f=\hat q^g$ here. We note that the improvement of the first bounds is particularly appreciable for small and large number of successes $n\Tb$. In particular, a nontrivial bound on $\qb$ is provided even in presence of a single success: $\hat q^f_\alpha(1,20)=\hat q^g_\alpha(1,20)=\frac{1}{400}$.
b) $n=100$ rounds and $n\Tb=10$ successes. $\hat q^{1}$ only applies when $\alpha\leq 1/2$~\cite{Munich19}. Hoeffding's confidence interval yields no conclusion for high confidence levels ($\alpha\lesssim 0.14$), whereas the other bounds are nontrivial as soon as $\alpha>0$.}
    \label{fig1}
\end{figure}

\subsection{Application to sequential sampling with and without independence between the samples}
Consider the sampling of $n$ binary variables $T_i$, $i=1,\ldots,n$ in a sequential fashion, i.e.~where the variables are sampled one after the other one: first $T_1$, then $T_2$, followed by $T_3$, etc. If the samples are independent from each other, it is possible to associate a winning probability $q_i\in[0,1]$ to each $T_i$, even before the start of the experiment. The confidence intervals given by $\hat q^f$ and $\hat q^g$ can then be applied directly to bound the average winning probability $\qb=\frac1n\sum_{i=1}^n q_i$ as a function of $n$, $n\Tb$ and $\alpha$. These bounds apply even if the random variables are not identically distributed, i.e.~if $q_i$ depends on $i$.

In the case where the samples are not guarantees to be independent from each other, however, the winning probability of each variable $T_i$ cannot be described anymore by a single parameter $q_i$ set a priori, i.e.~independently of the other random variables. Since the $n$ random variables may be correlated with each other in this case, they must be described by a general joint probability distribution of the form
\begin{equation}
P(T_1,\ldots,T_n).
\end{equation}
In a single instance of a sequential scenario in which $T_{i}$ is only observed after $T_{i-1}$, this joint probability distribution cannot be fully explored. Rather, the sampling of $T_i$ in this case is described by the conditional probability distribution
\begin{equation}
P(T_i|T_1=t_1,\ldots,T_{i-1}=t_{i-1}),
\end{equation}
where $t_1,\ldots,t_{i-1}=0,1$ are the results observed in the first $i-1$ samples of the experiment. In this situation, a difference arises between the average winning probability over both all rounds $i$ and multiple realizations of the experiment, and the average winning probability over all rounds $i$ for the single realization occurring in a specific experiment (unlike in the previous case of independent samples where both averages are equal). Interestingly, the second quantity can be evaluated precisely in a single run of the experiment. Indeed, this average winning probability over all rounds can be obtained by identifying the parameter $q_i$ with the winning probability of the $i^\text{th}$ random variable $T_i$ conditioned on the results of the previous samples, i.e.~defining
\begin{equation}
q_i = P(T_i=1|T_1=t_1,\ldots,T_{i-1}=t_{i-1}).
\end{equation}
The confidence intervals $\hat q^f$ and $\hat q^g$ described in this work can then be employed to bound the average winning probability $\qb$ of the variables that were actually sampled during the course of a specific one-shot non-i.i.d.~experiment.


\section{Acknowledgements}
We are thankful to Marius Junge and Nicolas Sangouard for insightful comments and to Davide Rusca and Hugo Zbinden for discussions. We also thank the University of Basel for hosting during part of this project.

\appendix

\section{Regularized Incomplete Beta Function}\label{sec:AppendixBetaFunction}
We recall that the regularized incomplete Beta function $I_x(a,b)$ is defined as~\cite{Wilks43}
\begin{equation}
I_x(a,b) = \frac{\Gamma(a+b)}{\Gamma(a)\Gamma(b)} \int_0^x t^{a-1}(1-t)^{b-1} dt
\end{equation}
with $a,b>0$, where $\Gamma(x)$ is the Gamma function, and that it is related to the cumulative distribution function of the binomial distribution by 
\be\label{eq: beta to binomial}\begin{split}
    \sum_{k=d}^n \binom{n}{k} x^k (1-x)^{n-k}&= I_{x}(d, n-d+1) \\
\sum_{k=0}^{d}\binom{n}{k} x^k (1-x)^{n-k} &=I_{1-x}(n-d,d+1).
\end{split}
\ee
Moreover, the function satisfies $I_x(a,b)=1-I_{1-x}(b,a)$~\cite{nist}.

\end{document}